\documentclass[12pt,a4paper]{article}

\usepackage{amsmath,amsthm,breqn}
\usepackage{graphicx}
\usepackage{natbib}
\newcommand{\blind}{1}

\usepackage{amssymb,bm,mathtools}
\usepackage{multirow}		% for tabular environment
\usepackage{braket}
\usepackage[hang,small,bf]{caption}
\usepackage[subrefformat=parens]{subcaption}
\usepackage{hyperref}
\usepackage{enumitem}

\newtheorem{thm}{Theorem}[section]
\newtheorem{lem}[thm]{Lemma}
\newtheorem{definition}[thm]{Definition}
\newtheorem{prop}[thm]{Proposition}
\newtheorem{corol}[thm]{Corollary}

\DeclareMathOperator{\ACT}{ACT}	% the adjusted correlation thresholding
\DeclareMathOperator*{\argmin}{arg\,min} % the argmin

\DeclareMathOperator{\BS}{BS}	% the broken-stick rule
\DeclareMathOperator{\Corr}{\mathbb{C}orr} % correlation of two random variables
\DeclareMathOperator{\Cov}{\mathbb{C}ov} % covariance of two random variables
\DeclareMathOperator{\diag}{diag}      % the diagonal matrix constructor
\DeclareMathOperator{\Exp}{\mathbb{E}} % the Expectation
\DeclareMathOperator{\ExpF}{\mathbb{E}_{|\bF}} % Conditional Expectation knowing \bF
\DeclareMathOperator{\MedianF}{\mathbb{M}_{|\bF}} % Conditional Median knowing \bF
\DeclareMathOperator{\rank}{rank}	 % the rank
\DeclareMathOperator{\trace}{trace}	 % the trace
\DeclarePairedDelimiter{\abs}{|}{|} % the absolute value
\DeclareMathOperator{\var}{\mathbb{V}} % the Variance
\DeclareMathOperator{\varF}{\mathbb{V}_{|\bF}} % Conditional Variance knowing \bF

\def\A{\mathbf{A}}		% a (real) matrix
\def\B{\mathbf{B}}		% a matrix
\def\b{\bm{b}}
\def\bDelta{\bm{\Delta}}	% a scaling diagonal matrix of a CLFM
\def\bell{\bm{\ell}}		% the factor loading vector of a CLFM
\def\bL{\mathbf{L}}     % the matrix of loading factor vectors
\def\bF{\mathbf{F}}		% the matrix of common factors of a CLFM
\def\bg{\mathbf{g}}		% \bg^i be the i-th vector of \bLambda
\def\bGamma{\bm{\Gamma}}	% the population coefficient of a CLFM
\def\bH{\mathbf{H}}		% a real symmetric matrix
\def\bmu{\bm{\mu}}		% the population mean of a CLFM
\def\bM{\mathbf{M}}		% the matrix with the columns being the population mean of a CLFM
\def\bmx{\bar{\x}}		% the sample average vector
\def\bN{\mathbf{N}}		% a Hermitian matrix
\def\bR{\mathbf{R}}		% the population correlation matrix
\def\bPsi{\mathbf{\Psi}}	% the idiosyncratic noise matrix
\def\bLambda{\mathbf{\Lambda}}	% the noise coefficient matrix
\def\bgamma{\bm{\gamma}}	% \bgamma^k is the $k$th row vector of $\bGamma$
\def\bSigma{\bm{\Sigma}}	% the population covariance matrix

\def\C{\mathbf{C}}		% the sample correlation matrix
\def\c{c}			% the dimension-to-sample-size ratio
\def\cas{\mathrel{\stackrel{a.s.}{\longrightarrow}}} % almost sure convergence

\def\D{\mathbf{D}}  
\def\dd{\mathrm{d}}		% exterior differentiation

\def\eit{\psi_{i \j}}		% the i-th idiosyncratic component at time \j

\def\fkt{f_{\k \j}}		       % factor
\def\fra#1#2{{#1}/{#2}}		       % fraction

\def\H{\mathbf{H}}		       % A Hermitian matrix
\def\harmonic#1{\sum_{n=#1}^\p n^{-1}}	% The sum of the inverses of #1 ... \p.
\def\hbPsi{\hat\bPsi}			% The re-centered truncation of
\def\hbx{\hat\x}			% A trend computed by HP-filter.
\def\heit{{\hpsi}_{i\j}}		% (i,\j)-entry of \hatbPsi
\def\hpsi{{\hat\psi}}
\def\I{\mathbf{I}}		% The identity matrix
\def\Indctr#1{\mathrm{I}\left(#1\right)}	 % The indicator function

\def\j{t}			 % index for time

\def\k{k}			 % variable ranging over [1, K]
\def\krr{\p/\n\to \c}		 % Kolmogorov regime ratio
\def\Krr{\frac{\p}{\n}\to\c}	 % Kolmogorov regime ratio (displaystyle)
\def\kr{{\p,\n\to\infty}}	 % Kolmogorov regime ratio (simple
\def\lam#1#2{\lambda_{#1}\left(#2\right)}
\def\lamLC#1{\lambda_{r+1}\left(#1\right)} % the largest bounded eigenvalue of $1$
\def\liminf{\varliminf}		% liminf
\def\limsup{\varlimsup}		% limsup

\def\M{\mathbf{M}}		% A Hermitian matrix

\def\mbF{\bar{\bF}}

\def\mbPsi{\bar{\bPsi}}

\def\mX{\bar{\X}}		

\def\mx{\bar{x}}		% \mx_i be the sample average of the $i$th row of $\X$
\def\mZ{\bar{\Z}}		% the time average of a random matrix \Z

\def\n{T}	   % the sample size
\def\norm#1{\left\|#1\right\|}	% the Euclidean norm

\def\o{\bm{0}}                  % null vector
\def\ones#1{\mathbf{1}_{#1}}

\def\P{\mathbf{P}}		% Probability
\def\p{N}			% the dimension of a data

\def\Rset{{\mathbb R}}    	% the set of real numbers
\def\rangei{{1\le i\le \p}}	% the range of an index i from 1 to \p
\def\rangej{{1\le\j\le \n}}	% the range of an index \j from 1 to \n
\def\rangek{{1\le \k\le K}}	% the range of an index k from 1 to K
\def\rangekr{{1\le \k\le r}}	% the range of an index k from 1 to K

\def\S{\mathbf{S}}		% the sample covariance matrix
\def\snorm#1{\left|\!\left|\!\left|#1\right|\!\right|\!\right|} % spectral norm
\def\sumj{\sum_{\j=1}^\n}	% the summation over \j
\def\summ{\sum_{\k=1}^K}	% the summation over k from 1 to K
\def\sv#1#2{s_{#1}\left(#2\right)} % the #1-th singular value of #2
\def\tbx{\tilde\x}      % trend 
\def\tC{\tilde\C}		% noncentered sample correlation matrix

\def\tD{\tilde\D}

\def\tS{{\tilde\S}}		% noncentered sample covariance matrix

\def\tx{\tilde x}      % component of trend 

\def\X{\mathbf{X}}         % the data matrix of a CLFM
\def\x{\bm{x}}		   % a vector associated to the data matrix \X
\def\xit{x_{i \j}}	   % the (i,\j)-entry of data matrix \X of a CLFM

\def\y{\bm{y}}		   % a time-series

\def\zit{z_{i \j}}	   % the (i,\j)-entry of a random matrix \Z of a CLFM
\def\Z{\mathbf{Z}}	   % the random matrix of a CLFM

\begin{document}

\def\spacingset#1{\renewcommand{\baselinestretch}%
{#1}\small\normalsize}

\if1\blind
{
  \title{\bf Asymptotic locations of bounded and unbounded eigenvalues of sample correlation matrices of certain factor models \\ -- application to a components retention rule}
  \author{Yohji Akama%\thanks{    The authors gratefully acknowledge}
\hspace{.2cm}\\
    The Mathematical Institute, Tohoku University, \\
    Aramaki, Aoba,
Sendai, 980-8578, Japan\\
     \\
   Peng Tian \\
    Laboratoire Jean Alexandre Dieudonn\'e, Universit\'e C\^ote
    d'Azur, \\
    28, Avenue Valrose, 06108 Nice Cedex 2, France}
  \maketitle
  } 
\fi

\if0\blind
{
  \bigskip
  \bigskip
  \bigskip
  \begin{center}
    {\LARGE\bf Asymptotic locations of bounded and unbounded eigenvalues of sample correlation matrices of certain factor models \\ -- application to a components retention rule}
\end{center}
  \medskip
} \fi

\bigskip
\begin{abstract}
Let the dimension $\p$ of data and the sample size $\n$ tend to $\infty$
with $\p/\n \to c > 0$.  The spectral properties of a sample correlation
matrix $\C$ and a sample covariance matrix $\S$ are asymptotically equal
whenever the population correlation matrix $\bR$ is
bounded~\citep{karoui09concentration}. We demonstrate this also for
general linear models for \emph{unbounded} $\bR$, by examining
the behavior of the singular values of multiplicatively perturbed
matrices.  By this, we establish: Given a factor model of an
idiosyncratic noise variance $\sigma^2$ and a rank-$r$ factor loading
matrix $\bL$ which rows all have common Euclidean norm $L$.  Then, the
$k$th largest eigenvalues $\lambda_k$ $(1\le k\le \p)$ of $\C$ satisfy
almost surely: (1) $\lambda_r$ diverges, (2) $\lambda_k/s_k^2\to1/(L^2 +
\sigma^2)$ $(1 \le k \le r)$ for the $k$th largest singular value $s_k$
of $\bL$, and (3) $\lambda_{r + 1}\to(1-\rho)(1+\sqrt{c})^2$ for $\rho
:= L^2/(L^2 + \sigma^2)$.  Whenever $s_r$ is much larger than
$\sqrt{\log\p}$, then broken-stick
rule~\citep{Frontier76:_etude,jackson1993stopping},
which estimates $\rank \bL$ by a random
partition (Holst, 1980) of $[0,\,1]$, tends to
$r$ (a.s.).
We also provide a natural factor model where the rule tends
 to ``essential rank'' of $\bL$ (a.s.) which is smaller than $\rank \bL$.
\end{abstract}

\noindent%

\section{Introduction}
\label{sec:intro}

In multivariate
analysis~\citep{anderson,fujikoshi2011multivariate,muirhead2009aspects},
covariance matrices are important objects, for example, in principal
component analysis (PCA).  Motivated by this, research on sample
covariance matrices has a long history and an abundance of results have
been established, e.g. various Wishart distributions.

Several well-known methods in multivariate analysis, however, become inefficient or even misleading when the data dimension $\p$ is large.
To deal with such large-dimensional data, a novel approach in asymptotic statistics has been developed where the data dimension $\p$ is no longer fixed but tends to infinity together with the sample size $\n$. 
The feature of this limiting regime are discussed in \cite[Section~1.1]{yao} based on real datasets such as portfolio, climate survey, speech analysis, face recognition, microarrays, signal detection, etc.
In this paper, we assume a high-dimensional regime in which the ratio $\p/\n$,
of dimension $\p$ to sample size $\n$, converges to a positive constant $c$.

One of milestone work about the spectral properties of sample covariance
matrices in this limiting regime was the discovery of
Mar\v{c}enko-Pastur distribution, and the
determination~\citep{YBK88,bai-yin2} of the largest and the smallest
eigenvalues of sample covariance matrices $\S$ for i.\@i.\@d.\@ data
under the finite fourth moment condition. \cite{baik06:_eigen} studied
asymptotic locations of the eigenvalues of $\S$ of a spiked eigenvalue
model~\citep{10.1214/aos/1009210544}, but the largest eigenvalues of the
population covariance matrix $\bSigma$ are bounded. We refer the reader
to \citep{bai2010spectral,PAUL20141,yao}.

In some cases, for example, when data are measured in different units, it is more appropriate to utilize sample correlation matrices~\citep{J,jolliffe2016principal,JohnsonWichern} because sample correlation matrices $\C$ are invariant under scaling and shifting. 
Moreover, the assumption of i.\@i.\@d.\@ data is not very acceptable; Practitioners often hope to be in the presence of a curious covariance structure. 
However, the asymptotic spectral properties of a sample correlation matrix $\C$ in the high-dimensional regime have not been sufficiently investigated, compared to sample covariance matrices $\S$.

We review typical asymptotic results of the spectral properties of sample correlation matrices $\C$ here. 
For i.\@i.\@d.\@ case, under the finite fourth moment condition,
\cite{10.2307/25053330} showed that the extreme
eigenvalues of the sample correlation matrix $\C$ converge almost
surely to $(1\pm \sqrt{c})^2$, and \cite{bai-zhou} showed that the limiting spectral distribution is
the standard Mar\v{c}enko-Pastur distribution under a weaker assumption.

In a class of spiked models, \cite{morales2021asymptotics} derived asymptotic first-order and distributional results for spiked eigenvalues and eigenvectors of sample correlation
matrices $\C$. More specifically, 
they found that the first-order spectral properties of sample correlation matrices $\C$ match those of sample covariance matrices $\S$, whilst their asymptotic distributions can differ significantly. 

\cite{karoui09concentration} revealed that the first-order asymptotic behavior of the spectra of $\C$ is similar to that of $\S$, for unit variance data of a general linear model, except that this similarity requires the boundedness of the population correlation matrix $\bR$, a condition not met by some factor models. 

The problem is that the boundedness of the population covariance matrices $\bSigma$ is not
always satisfied in econometrics,
finance~\citep{chamberlain83:_arbit_factor_struc_mean_varian,bai02:_deter_the_number_offac_in},
genomics, and stationary long-memory processes. One of the features of the
unbounded population covariance matrices $\bSigma$ is the consistency of the eigenvectors~\citep{yata13:_pca,MR3556768,10.1214/16-AOS1487}, although the asymptotic location and the fluctuation of the
largest unbounded eigenvalues of $\C$ are not available yet even for the equi-correlated normal population, which has the simplest unbounded population covariance matrix $\bSigma$.

Unbounded covariance/correlation matrices in high-dimensional problems are studied with a spiked model~\citep{yata13:_pca}, a time-series model~\citep{Merlev_de_2019} or a factor model~\citep{10.1214/18-AOS1798,10.1214/16-AOS1487}.
Following the latter articles for unbounded sample covariance matrices, we study factor models for unbounded sample correlation matrices. 

Our target model is a \emph{$K$-factor model} 
$$\X=\begin{bmatrix} \bmu &\dots & \bmu
\end{bmatrix} + \bL\bF + \bLambda\bPsi$$ 
with $\bmu\in\Rset^\p$, $\bL\in\Rset^{\p\times K}$, $\bLambda\in\Rset^{\p\times\p}$ being deterministic, $K$ a fixed positive integer, $\bF=[f_{i\j}]\in\Rset^{K\times\n}$  and $\bPsi=[\psi_{i\j}]\in\Rset^{\p\times\n}$ two random matrices.
Here we assume that the entries of $\bF$ (\emph{factors}) are i.i.d.\@ centered random variables with unit variance, and so are the entries of $\bPsi$ (\emph{noises}). The entries of $\bF$ are independent from the entries of $\bPsi$, but \emph{$f_{11}$ and $\psi_{11}$ are not necessarily identically distributed}. All rows of $\bGamma:= \begin{bmatrix}
    \bL & \bLambda
\end{bmatrix}$ are nonzero.

Our fundamental result is Theorem~\ref{thm:asymptotic equi-spectral}: for every $K$-factor model, 
if the entries $\eit$ of the noise matrix $\bPsi$ have finite fourth moments
and $\bLambda$ is diagonal, or if there exists $\varepsilon>0$ such that $\Exp\left(\,|\eit|^4(\log|\eit|)^{2+2\varepsilon}\right)<\infty$, 
then the eigenvalues of the sample correlation
matrices $\C$ and those of the sample covariance matrices $\S$ with unit-variance data are asymptotically equal, which extends the result of \cite{karoui09concentration}.

In PCA, methods for estimating the number of factors in a sample have been
studied
by~\citep{bai02:_deter_the_number_offac_in,10.1214/12-AOS970,ait2017using}
in econometrics, etc. 

Meanwhile, rules for the retention of principal components have been proposed in many lectures (see, e.g., \citep{jackson1993stopping}). \emph{Broken-stick (BS)
rule}~\citep{Frontier76:_etude,jackson1993stopping} is
a peculiar rule among such rules. BS rule compares the spectral
distribution of $\C$, with a distribution of the mean lengths of the
subintervals obtained by a random partition of [0,\, 1]~\citep{Holst80}. Specifically, for the number of factors or
significant principal components, the BS rule provides $i-1$ where $i$
is the lowest among $[1,\,\p]$ for which the $i$th largest eigenvalue of
$\C$ does not exceed the sum $1/i + 1/(i+1) + \cdots + 1/\p$.  The BS
rule depends on the number and growth rates of \emph{unbounded}
eigenvalues of $\C$. 
The idea of the BS rule, initially coming from the species occupation model in an ecological system, is not evident in relation to the distribution of eigenvalues of $\C$.

First, we study the asymptotic spectral properties of $\C$, through the fundamental theorem of $K$-factor model~(Theorem~\ref{thm:asymptotic equi-spectral}) and techniques from the random matrix theory, for two illustrative $K$-factor models generalizing the equi-correlated normal population. 
One of the two is a $K$-factor model such that $\bLambda=\sigma\I_\p$ and the rows of the factor loading matrix $\bL$ have a common length. We call this a \emph{constant length factor loading model}~(CLFM). 
The other is a $K$-factor model such that the rows of $\bL$ and the diagonal entries of $\bLambda$ are convergent. This model was introduced in \cite{Akama23:_correl}, and is called an \emph{asymptotic convergent factor model}~(ACFM) here.
We establish that under certain mild conditions, the BS rule precisely matches the principal rank of the factor loading matrix $\bL$ for a CLFM (see Theorem~\ref{thm:CLFM_BS}), but the \emph{essential rank} of $\bL$ for an ACFM (see Theorem~\ref{thm:asympt_CLF}). 

Then we calculate the BS
rule and some modern factor number estimators such as the
\emph{adjusted correlation thresholding}~\citep{FGZ22} and Bai-Ng's rule~\citep{bai02:_deter_the_number_offac_in}
based on an information criterion, for financial datasets and a biological dataset~\citep{querde}. 
The financial datasets 
obtained by \cite{FGZ22} from Fama-French 100 portfolio~\citep{FAMA19933} by cleaning-up outliers.  
The biological dataset is a binary multiple sequence alignment~(MSA) of HCV genotype 1a (prevalent in North America) publicly available from the Los Alamos National Laboratory
database~\citep{10.1093/bioinformatics/bth485}.

\iffalse
For an arbitrary CLFM, we establish the following: If the entries of $\bPsi$ have the finite fourth moments, and the nonzero singular value of $\bL$ is divergent, then the $k$th eigenvalue of $\C$ over the squared $k$th largest singular value of $\bL$ converges almost surely to $1/(L^2+\sigma^2)$ for $1\le k\le r:=\rank \bL$, and that the
 $(r+1)$st largest eigenvalue of $\C$ converges almost surely to
 $\sigma^2(1+\sqrt{c})^2/(L^2+\sigma^2)$. If furthermore the squared $r$th singular
 value of $\bL$ grows faster than $\log\p$, then the \emph{broken-stick
 rule}~\citep{Frontier76:_etude,jackson1993stopping}, 
 
 For an arbitrary ACFM, we demonstrate the following: if all the rows of $\bGamma$ are nonzero, the entries of
 $\bPsi$ have finite fourth moments, and the sum of the squared Euclidean
 errors of
 the $i$th rows of $\bL$ from the limiting vector $\bell$ $(1\le i\le \p)$ is $o(\log\p)$,
 then the broken-stick rule converges almost surely to the indicator
 function of $\bell\ne\o$.
\fi
 
The structure of this paper is as follows. Section~\ref{sec:model}
formally introduces the model and presents the main theorems, with
proofs provided in Section~\ref{sec:proof_theorems} of the supplementary
material. Numerical analysis is performed on actual financial and
biological datasets in Section~\ref{sec:empirical
study}. Section~\ref{sec:conclusion} is the
conclusion. Section~\ref{sec:useful lemma} of the supplementary material
includes several important lemmas derived from general random matrix
theories.

\section{Models and theoretical results} \label{sec:model}
In this paper we consider the following model of data matrix:
\begin{definition}[$K$-Factor model] \label{def:K-FM}Let
 $K,\p,\n$ be nonnegative integers. A $K$-factor model is a random matrix $\X \in{\Rset}^{\p\times \n}$ in the form 
\begin{align*}
    \X := \bM + \bL\bF+\bLambda \bPsi,
\end{align*}
    where
    \begin{enumerate}
        \item \label{def:bM} the deterministic matrix $\bM=\begin{bmatrix}
    \bmu &\dots & \bmu
\end{bmatrix}\in\Rset^{\p\times \n}$ is the {\em theoretical mean}, with $\bmu\in\Rset^\p$;
        \item the deterministic matrices $\bL\in\Rset^{\p \times K}$ and $\bLambda\in\Rset^{\p\times \p}$ are called the {\em factor loading matrix} and {\em noise coefficient matrix}, and the rows of\/ $\bL$ are called {\em factor loading vectors};
        \item there are two independent sets of i.\@i.\@d.\@ random variables $\{f_{i\j}\}_{1\le i\le K,\j\ge 1}$ and $\{\eit\}_{i,\j\ge 1}$ with mean zero and variance one, such that
        $$\bF=[f_{i\j}]_{1\le i\le K,1\le \j\le \n},\quad \bPsi=[\eit]_{1\le i\le \p,1\le \j\le \n}.$$
        The two random matrices $\bF$ and $\bPsi$ are called {\em factor matrix} and {\em idiosyncratic noise matrix}, respectively.
    \end{enumerate}
\end{definition}

We often write $\X$ in a compact form
$$\X = \bM + \bGamma\Z$$
with
$$\bGamma = \begin{bmatrix}
    \bL & \bLambda
\end{bmatrix}\quad \text{and}\quad \Z=\begin{bmatrix} \bF \\ \bPsi \end{bmatrix}.$$
Let $\x_j$ be the $j$th column of $\X$. Then $\x_1,\dots,\x_\n$ are i.\@i.\@d.\@ random vectors. We recall that the population covariance and correlation matrices are $[\Cov(x_{i1},x_{k1})]_{\p\times \p}$ and $[\Corr(x_{i1},x_{k1})]_{\p\times \p}$, respectively, where $x_{ij}$ is the $i$th component of the vector $\x_j$. Note that under Definition~\ref{def:K-FM}, the population covariance matrix is
$$\bSigma:=\bGamma\bGamma^\top = \bL\bL^\top+\bLambda\bLambda^\top,$$
and the population correlation matrix is 
$$\bR:=\bDelta^{-\frac{1}{2}}\bSigma\bDelta^{-\frac{1}{2}},$$
where 
\begin{align*}\bDelta=\diag(\var(x_{11}),\dots,\var(x_{\p 1}))
\end{align*}
is a diagonal matrix containing the variance of components of $\x_1$. Note that $\var(x_{k1})$ is just the square of Euclidean norm of the $k$th row vector $\bgamma^k$ of $\bGamma$, and is also the $k$th diagonal element of $\bSigma$. So we can write $\bDelta=\diag(\bSigma)$, where for a square matrix $\A$, $\diag(\A)$ denotes the diagonal matrix with the same diagonal as $\A$.

The population covariance matrix $\bSigma$ and correlation matrix $\bR$ play important roles in multivariate statistics. However they are not always available. In order to estimate them, we define the {\em theoretically-centered} sample covariance matrix $\tS$ as
\begin{align}
\tS&=\frac{1}{\n}(\X-\bM)(\X-\bM)^\top.\label{eq:def_tS}\end{align}
Let $\tD=\diag(\tS)$. Then the {\em theoretically-centered} sample correlation matrix $\tC$ is defined as (see e.g. \cite{karoui09concentration})
\begin{align*}
\tC&=\tD^{-\frac{1}{2}}\tS \tD^{-\frac{1}{2}}.
\end{align*}

As the mean vector $\bmu$ is not always known either, we sometimes need to replace $\bmu$ by the sample mean 
\begin{align*}\bmx=\frac{1}{\n}\sumj \x_\j,\end{align*}
and use it to define the {\em data-centered} sample covariance matrix
\begin{align*}
\S=\frac{1}{\n-1}(\X-\mX)(\X-\mX)^\top, 
\end{align*}
where $\mX=\begin{bmatrix}\bmx & \cdots &\bmx\end{bmatrix}_{\p\times \n}$, and the {\em data-centered} sample correlation matrix 
\begin{align*}
   \C=\D^{-\frac{1}{2}}\S\D^{-\frac{1}{2}}, 
\end{align*}
where $\D=\diag(\S)$.

In this paper we focus on the sample correlation matrices. Note that if one row of $\bGamma$ is identically $0$, then the corresponding row of $\X$ is deterministically equal to the mean, and the correlation between this row and any other row is not defined. But it is easy to recognize such a row in the data and to eliminate it before further treatment. So we can assume 
\begin{enumerate}[series=assumptions,leftmargin=*, label={\bf A\arabic{enumi}}]
    \item \label{ass:Gamma_row_ne0}The rows of $\bGamma$ are nonzero:
$$\bgamma^i \ne \o, \quad i=1,\dots,\p.$$
\end{enumerate}

We study the limiting locations of eigenvalues of sample correlation matrices $\C$ and $\tC$ in the \emph{proportional limiting regime} 
\begin{align*}
 \kr,\ \Krr > 0,
\end{align*}
which will be denoted as
$$\kr$$
for simplicity. In this regime, Theorem~1 in \cite{karoui09concentration} related the spectral properties (limiting spectral distributions and limit locations of individual eigenvalues) of $\C$ and $\tC$ to those of sample covariance matrices $\bDelta^{-\frac{1}{2}}\S\bDelta^{-\frac{1}{2}}$ and $\bDelta^{-\frac{1}{2}}\tS\bDelta^{-\frac{1}{2}}$, in condition that $\snorm{\bDelta^{-\frac{1}{2}}\bGamma}$, the spectral norm of $\bDelta^{-\frac{1}{2}}\bGamma$, are uniformly bounded. However, due to the presence of $\bL$, this condition is not satisfied by some factor models. For example, let us consider an ENP studied in~\citep{fan19:_larges,AH22,Akama23:_correl} and defined in the following Definition~\ref{def:ENP}. Then $\snorm{\bDelta^{-\frac{1}{2}}\bGamma}=\sqrt{(L^2\p+\sigma^2)/(L^2+\sigma^2)}$ which diverges to $\infty$ as $\p\to\infty$.
\begin{definition}
    \label{def:ENP}
An \emph{equi-correlated normal population} (ENP for short) is a $1$-factor model such that (1) $\bL=[L \cdots L]^\top\in\Rset^{\p}$ for some $L>0$; and (2) $\bLambda = \sigma \I$ for some $\sigma>0$, and (3) $\fkt$ $(\rangek=1)$ and $\eit$ $(\rangei,\, \rangej)$ are independent standard normal random variables. For an ENP, we define $\rho=L^2/(L^2+\sigma^2)$.
\end{definition}

In this paper, for the $i$th  largest singular value $\sv{i}{\A}$ of a matrix
$\A$, we prove the following Theorem in
Section~\ref{sec:proof_theorems}:

\begin{thm}\label{thm:perturbation_multiplicative}Let $\A$ and $\B$ be complex matrices of order $m\times n$ and $n\times p$. 
\begin{align*}\sv{m}{\A}\sv{i}{\B}&\le \sv{i}{\A\B}\le \sv{1}{\A}\sv{i}{\B}&(i\ge1).\end{align*}
\end{thm}
Thanks to this theorem, we managed to generalize Theorem~1 in \cite{karoui09concentration} to general $K$-factor models with possibly unbounded $\bDelta^{-\frac{1}{2}}\bGamma$. 

Before stating our first main result, we add some additional assumptions.
\begin{enumerate}[series=assumptions,resume,leftmargin=*, label={\bf A\arabic{enumi}}]
    \item \label{ass:moment4} Each random variable $\eit$ has finite forth moment:
$$\Exp\left(|\eit|^4\right)<\infty.$$
    \item \label{ass:karoui_cond_or_identity}One of the following assumptions holds: (a) There exists $\varepsilon>0$ such that 
$$\Exp\left(\,|\eit|^4(\log|\eit|)^{2+2\varepsilon}\right)<\infty,$$ 
    or (b) $\bLambda$ is diagonal.
\end{enumerate}

For two nonnegative sequences $\{a_n\}$ and $\{b_n\}$, by $a_n\sim b_n$ we mean that there is a positive sequence $\{\tau_n\in (0,1)\}$ such that $\lim_{n\to\infty}\tau_n=1$, and for large enough $n$, $\tau_n a_n\le b_n\le \tau_n^{-1}a_n$. For a family of such sequences $\{a_n^{(i)}\}_i$ and $\{b_n^{(i)}\}_i$, we say that $a_n^{(i)}\sim b_n^{(i)}$ uniformly if there is $\{\tau_n\in (0,1)\}$ independent of $i$, with $\lim_{n\to\infty}\tau_n=1$, such that $\tau_n a_n^{(i)}\le b_n^{(i)}\le \tau_n^{-1}a_n^{(i)}$ holds for all $i$ and for large enough $n$.

\begin{thm}\label{thm:asymptotic equi-spectral}For a $K$-factor model with $K\ge 0$ fixed, if \ref{ass:Gamma_row_ne0}-\ref{ass:karoui_cond_or_identity} hold, then as $\kr$, $\krr>0$, for any $i=1,2,\dots,\p$, we have almost surely
$$\lambda_i(\C)\sim\lambda_i(\bDelta^{-\frac{1}{2}}\S\bDelta^{-\frac{1}{2}}), \quad \lambda_i(\tC)\sim\lambda_i(\bDelta^{-\frac{1}{2}}\tS\bDelta^{-\frac{1}{2}})$$
uniformly.
\end{thm}

It should be noted that no bounding condition for $\bDelta^{-1/2}\bLambda$ is required. Therefore, when $\bL=\o$, the theorem is applicable to a general linear model $\X=\M+\bLambda\bPsi$, where $\bLambda$ can be unbounded. Furthermore, assuming $\bL \neq \o$, we accommodate different distributions for factor and noise components, for example, the factors are allowed to have a heavy-tailed distribution, along with a light-tailed noise.

The general result being stated, it is still complex to determine the asymptotic locations of a sample covariance matrix in general. We content ourselves with some particular cases that extend an ENP.

\paragraph{Model example 1: Constant length factor loading model (CLFM).}

\begin{definition}[Constant length factor loading model]\label{def:CLFM} By a \emph{constant length factor loading model}~$($CLFM for short$)$, we mean a $K$-factor model with $\bLambda=\sigma\I$ for some $\sigma>0$, whose factor loading vectors have the same length, i.e., there is a constant $L\ge0$ independent of $\p$ and $\n$ such that the $i$th row vector $\bell^i$ $(\rangei)$ of $\bL$ has length $\norm{\bell^i}=L$.
\end{definition}
Note that a CLFM satisfies automatically \ref{ass:Gamma_row_ne0}, 
and \ref{ass:karoui_cond_or_identity}(b). Moreover, the largest eigenvalue of $\bL^\top\bL$ has asymptotic tight order $\p$:
\begin{align}
\frac{L^2\p}{K}\le \lambda_1(\bL^\top\bL)\le L^2\p.\label{order:largest_eig_LLT}
\end{align}
It is because, by letting $\bell_\k$ be the $k$th column of $\bL$ for $\rangek$, we have
$$\max_{\rangek}{\norm{\bell_\k}^2}\le \lambda_1(\bL^\top \bL)\le
 \trace(\bL^\top \bL)=L^2 N\le K\max_{\rangek}{\norm{\bell_\k}^2}$$ where
the first inequality is due to Courant-Fisher min-max
theorem~\cite[Theorem~4.2.6]{horn2013}, and the equality between the
trace and $L^2\p$ is due to that the length of each row of $\bL$ is $L$.   

Furthermore, we will consider the following assumption:
\begin{enumerate}[series=assumptions,resume,leftmargin=*, label={\bf A\arabic{enumi}}]
    \item \label{ass:L_rank_r_sr_infty} The rank of $\bL$ is $r\ge 0$, and the $r$ nonzero eigenvalues of $\bL^\top \bL$ (if any) tend to infinity: for $k=1,\dots,r$,
    \begin{align*} \lim_{\p\to\infty}\lambda_k(\bL^\top\bL)=\infty. 
    \end{align*}
\end{enumerate}

\begin{thm}\label{thm:main_CLFM_largest}
If a CLFM satisfies \ref{ass:moment4} and \ref{ass:L_rank_r_sr_infty},
then for $k=1,\ldots,r$,
 \begin{align}
  \lim_\kr\frac{\lam{k}{\C}}{\lam{k}{\bL\bL^\top}}=\lim_\kr\frac{\lambda_{k}(\tC)}{\lam{k}{\bL\bL^\top}}&=\frac{1}{L^2+\sigma^2} & (a.s.),\label{thm:main_CLFM_unbounded}
 \end{align}
 and
\begin{align}
\lim_\kr\lamLC{\C}=\lim_\kr\lambda_{r+1}(\tC)&=\frac{\sigma^2(1+\sqrt{\c})^2}{L^2+\sigma^2} & (a.s.).\label{thm:main_CLFM_bounded}
\end{align}
\end{thm}

If a given CLFM is an ENP, then Definition~\ref{def:ENP} implies $K=1$, which gives $\lam{1}{\bL\bL^\top}=L^2\p$ by \eqref{order:largest_eig_LLT}. Thus, by Theorem~\ref{thm:main_CLFM_largest}~\eqref{thm:main_CLFM_unbounded}, both $\lim_\kr\lam{1}{\C}/\p$ and $\lim_\kr\lambda_1(\tC)/\p$ are almost surely $\rho$, which is proved in \cite{Akama23:_correl}. 
Theorem~\ref{thm:main_CLFM_largest}~\eqref{thm:main_CLFM_bounded} implies that $\lim_\kr\lamLC{\C}=\lim_\kr\lambda_{r+1}(\tC)=(1-\rho)(1+\sqrt{c})^2$ almost surely.

As a direct application of Theorem~\ref{thm:main_CLFM_largest}, we establish that the limit of the broken-stick rule $\BS(\C)$ and $\BS(\tC)$ equal to $r$. 

\begin{thm}\label{thm:CLFM_BS}
If a CLFM satisfies \ref{ass:moment4}, and if the smallest nonzero eigenvalue (if there exists) of $\bL^\top\bL$ is eventually larger than $\log\p$, i.e.,
\begin{align}
\varliminf_{\p\to\infty}\frac{\lambda_r(\bL^\top\bL)}{\log \p} > L^2+\sigma^2 , \label{eq:lambda_r_underbound}
\end{align}
then
\begin{align*}
     \lim_\kr\BS(\C)&=\lim_\kr\BS(\tC)=r &(a.s.).
\end{align*}
\end{thm}

Here we give a sufficient condition to ensure \ref{ass:L_rank_r_sr_infty} or \eqref{eq:lambda_r_underbound} above. If the rank of $\bL$ is $r\le K$, then by rearranging the columns, $\bL$ can be written as
$$\bL=\begin{bmatrix}
    \bL_1 & \bL_1\mathbf{A}
\end{bmatrix}$$
where $\bL_1$ is an $\p\times r$ matrix and $\mathbf{A}$ an $r\times (K-r)$ matrix. Then by eigenvalue interlacing theorem~\citep[Theorem~4.3.28]{horn2013},
$$\lambda_r(\bL\bL^\top)\ge \lambda_r(\bL_1\bL_1^\top).$$
Let $\B=\diag(\norm{\bell_1},\ldots,\norm{\bell_r})$ and
$$\cos(\bell_i,\bell_j)=\frac{\langle
\bell_i,\bell_j\rangle}{\norm{\bell_i}\norm{\bell_j}}$$ be the cosine of the angle generated by the two vectors $\bell_i,\bell_j$. 
By virtue of $\bL_1^\top\bL_1=\B\,\left[\cos(\bell_i,\bell_j)\right]_{1\le i,j\le r}\,\B,$
 Theorem~\ref{thm:perturbation_multiplicative} implies
$$\lambda_r(\bL_1\bL_1^\top)\ge 
\lambda_r\left(\,\left[\cos(\bell_i,\bell_j)\right]_{1\le i,j\le r}\right)
\min_{1\le j\le r}\norm{\bell_j}^2.
$$
By \cite[Theorem~1]{varah1975lower},
\begin{align*}
\lambda_r\left(\,\left[\cos(\bell_i,\bell_j)\right]_{1\le i,j\le r}\right)\ge \min_{1\le \k\le r}\left(1-\sum_{j\ne \k}|\cos(\bell_j,\bell_\k)|\right).
\end{align*}

Therefore, for any positive sequence $\{b_\p>0\}$, 
      \begin{align}
       \varliminf_{\p\to\infty} \frac{\min_{1\le k\le r}\left(1-\sum_{j\ne \k}|\cos(\bell_j,\bell_\k)|\right)\min_{1\le k\le r}\norm{\bell_\k}^2}{b_\p} > 1 \label{eq:lowerbound_lambda_r_LL}
      \end{align}
is a sufficient condition of 
$$\varliminf_{\p\to\infty}\frac{\lambda_r(\bL\bL^\top)}{b_\p} > 1.$$

This condition is noteworthy because the magnitude of $\norm{\bell_j}$ can be seen as the overall influence of the $j$th factor $\{f_{jt}\}_t$ on the dataset; whereas the normalized vector $\bell_j/\norm{\bell_j} \in \Rset^\p$ indicates the effects of the $j$th factor $\{f_{jt}\}_t$ across the $\p$ covariates. By Condition~\eqref{eq:lowerbound_lambda_r_LL}, if $r$ factors exert equally substantial influences on the dataset and their impact distributions are sufficiently independent, then the value of $\lambda_r(\bL\bL^\top)$ tends to be large.

\paragraph{Model example 2: Asymptotic constant factor loading model (ACFM).} We consider a $K$-factor model satisfying the following 
\begin{enumerate}[series=assumptions,resume,leftmargin=*, label={\bf A\arabic{enumi}}]
    \item \label{ass:limiting_CLFM} There is a sequence of $K$-dimensional row vectors $(\bell^k)_{k\ge 1}$ and a sequence of strict positive numbers $(\sigma_k)$ such that 
    $$\bL=\begin{bmatrix}
        (\bell^1)^\top & \dots (\bell^\p)^\top
    \end{bmatrix}^\top, \quad \bLambda=\diag(\sigma_1,\dots,\sigma_\p),$$
    and there are $\bell\in \Rset^{K}$ and $\sigma>0$ such that 
    $$\lim_{k\to\infty}\norm{\bell^k-\bell}=0;\quad \lim_{k\to\infty}\sigma_k=\sigma.$$
\end{enumerate}

This model was considered in \cite{Akama23:_correl} where the limiting spectral distribution of its sample correlation matrices was derived. The limit of the broken-stick rule $\BS(\C)$ and $\BS(\tC)$ for this model is determined here. 
\begin{thm}\label{thm:asympt_CLF} If a $K$-factor model satisfies \ref{ass:Gamma_row_ne0}, \ref{ass:moment4} and \ref{ass:limiting_CLFM}, and if in additional,
$$\sum_{k=1}^\p \norm{\bell^k-\bell}^2=o(\log \p),$$
then
    $$\lim_{\kr}\BS(\C)=\lim_{\kr}\BS(\tC)=\Indctr{\bell\ne \o}, \quad (a.s.),$$
where $\Indctr{\cdot}$ is the indicator function.
\end{thm}

An equi-correlated normal population (ENP) is a CLFM and an ACLM, but is not
a model of \cite{FGZ22}; the population covariance matrix of an ENP is $\Sigma = [\rho] + \diag(1-\rho)$  for some constant $\rho$, so an ENP does not satisfy the condition C3 of Section 2 ``High-Dimensional Factor Model'' of \citep{FGZ22}. Some model of \cite{FGZ22} is neither a CLFM nor an ACLM, because of the definitions of CLFM and ACLM.

The proofs of Theorems~\ref{thm:perturbation_multiplicative}, \ref{thm:asymptotic equi-spectral}, \ref{thm:main_CLFM_largest}, \ref{thm:CLFM_BS} and \ref{thm:asympt_CLF} are provided in  Section~\ref{sec:proof_theorems}.

\section{Broken-stick rule for real datasets}\label{sec:empirical study}
We will check whether the following real datasets are generated by a CLFM or an ACFM,  based
on Theorems~\ref{thm:main_CLFM_largest} and \ref{thm:asympt_CLF}:
\begin{enumerate}

\item the datasets~\cite{FGZ22} obtained by cleaning outliers from
       the datasets of the daily excess returns~\citep{https://doi.org/10.1111/j.1540-6261.1968.tb00815.x} of Fama-French 100
       portfolios~(\cite{FAMA19933,FAMA20151}. See Prof.\@ French's data library \url{http://mba.tuck.dartmouth.edu/pages/faculty/ken.french/}).

\item 	binary multiple sequence
	alignment~\citep{querde} by
	the courtesy of Prof. Quadeer.
 
\end{enumerate}
For the correlation matrices $\C$ of these datasets, we will compute the
eight quantities:
\begin{enumerate}
 \item the sample size $\n$,
 \item $\p/\n$,

 \item        $\lam{1}{\C}/\p$
       (cf.\@ Theorem~\ref{thm:main_CLFM_largest}~\eqref{thm:main_CLFM_unbounded}),
 \item the dimension $\p$,
\item the broken-stick rule $\BS(\C)$,

\item \emph{Adjusted correlation
thresholding}~\citep{FGZ22} $
\ACT(\C)$. 

 \item Bai-Ng's rule based on information criterion~\citep{bai02:_deter_the_number_offac_in}.
\end{enumerate}

In a stock return dataset, $\p$~($\n$, resp.) intends the
number of companies~(trading days$-1$, resp.). For
$i,\j$~$(\rangei,\rangej)$, the $i$th row of a data matrix $\X$ is for
the $i$th company and $\j$th column is for the $\j$th trading day.
The factors are those of Fama-French, for instance.

\subsection{Stock return time-series}\label{subsec:stock datasets}
\cite{de} proposed their \emph{Dynamic Equicorrelation} to
forecast time-series of economics.
In asset pricing and portfolio management, Fama and French designed
statistical models, namely, 3-factor model~\citep{FAMA19933} and
5-factor model~\citep{FAMA20151}, to describe stock returns.

\cite{FGZ22} computed their estimator ACT and confirmed the three
Fama-French factors~\citep{FAMA19933}, for the following datasets of the daily excess returns of 100 companies French chose.
\begin{enumerate}
 \item  Dot-com bubble \& recession~(1998-01-02/2007-12-31).

      In this   case $\p=100$ and $\n=2514$. (``Before 2007-2008 financial crisis'')

 \item After Lehman shock~(2010-01-04/2019-04-30).

In this   case $\p=100$ and $\n=2346$. (``After 2007-2008 financial crisis'')
\end{enumerate}

\begin{table}[ht] 
\centering
\begin{tabular}{ |c|c|c|l|c||c|c|c|}
\hline
Fama-French Portfolios & $\n$ &  $\fra{\p}{\n}$ & \small $\frac{\lam{1}{\C}}{\p}$
 &$\p$ &\small $\BS$&{\small$\ACT$}&Bai-Ng
\\
\hline
{\small 1998/2007~(Dot-com bubble \& recession)}& 2514 & .0398 & .658 & \multirow{2}*{100} & 2&4 & 4% 2
  \\  % std
				       %.6016 \\ norm
{\small 2010/2019~(After Lehman shock)}& 2346 & .0426 & .806 &     & 1 &3 & 5% 2
 \\ % std
				      % .7511\\ norm
 \hline
\end{tabular}
 \caption{The stock return datasets of the two Fama-French portfolios (1998-01-02/2007-12-31, 2010-01-04/2019-04-30).\label{tbl:FF}}
\end{table}

\begin{figure}[htbp]
  \begin{tabular}{c|c}
 \centering
 \begin{subfigure}[b]{0.47\linewidth}
    \centering
    \includegraphics[keepaspectratio, scale=0.45]{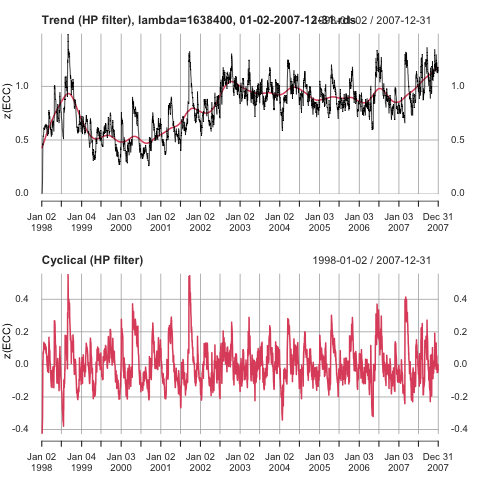}
\end{subfigure}
&  \begin{subfigure}[b]{0.47\linewidth}
    \centering
    \includegraphics[keepaspectratio, scale=0.45]{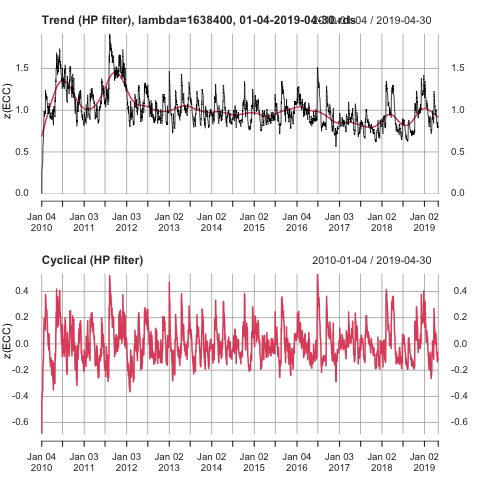}
  \end{subfigure}
\end{tabular}   
  \caption{For 1998-01-02/2007-12-31 (left) and  2010-01-04/2019-04-30 (right), Fisher's
 z-transform of \emph{equi-correlation coefficient} are computed by GJR
 GARCH+Engle-Kelly's dynamic equicorrelation. Their cyclical components
 (lower paletts) is computed by Hodrick-Prescott filter, and the residues (trends)
are overlaid in the upper paletts. \label{fig:FF}}
\end{figure}

As we see in Figure~\ref{fig:FF},
equicorrelation coefficient is low during dot-com bubble \emph{by speculation to  uncorrelated emergent stocks.}

We have two large groups of less speculative companies: 
\begin{itemize}
 \item the group of non-dot-com companies, in the dot-com bubble.
 \item the totality, in the  dot-com recession.
\end{itemize} 

Figure~\ref{fig:FF} is produced by the combination of a few techniques.
\begin{enumerate}
\item GJR GARCH with the correlation model being the Dynamic Equicorrelation~\citep{de}.
      
This computes the time-series of equicorrelation coefficient $\varrho$.
Our program utilizes \texttt{dccmidas}
package~\citep{candila21:_packag_dcc_model_garch_midas} of \texttt{R}.

 \item Fisher's z-transform $\frac{1}{2}\log\frac{1+\varrho}{1-\varrho}$.

       This bijection between $(-1,\, 1)$ and $(-\infty,\,\infty)$ makes
       the movement of $\varrho$ clearer. By z(ECC) time-series, we mean
       the time-series of the z-transform of $\varrho$. Fisher's
       z-transform often obeys
     asymptotically a normal distribution. In Figure~\ref{fig:FF},
 the z(ECC) time-series are on the upper paletts.
      
 \item  \emph{Hodrick-Prescott filter}~\citep{hodrick97:_postw_u} to
	compute the trend and the cyclical component of the z(ECC)
	time-series.
	
 From a given
 time-series $\y=[y_\j]_{\rangej}\in\Rset^{\n}$, \emph{Hodrick-Prescott filter of
		 parameter $\lambda>0$} extracts $$\tbx=[\tx_\j]_{\rangej}=\argmin_{\x\in\Rset^\n}\left(\norm{\y-\x}^2 +
\lambda\norm{\partial^2 \x}^2\right).$$ 
Here $\partial^2\x$ is the second
difference $\left[x_\j-2 x_{\j+1} + x_{\j+2}\right]_{1\le\j\le \n-2}\in\Rset^{\n-2}$ of $\x=\left[x_\j \right]_{\rangej}\in\Rset^\n$. Then, $\tx_\j$'s are weighted averages of
$y_1,\ldots,y_\n$ and satisfy
$\sumj(y_\j-\tx_\j)=0$. $\lim_{\lambda\to\infty}\tbx$ is the
arithmetical progression obtained from $\y$ by the least square
approach, but $\tbx$ is $\y$ for $\lambda=0$. $\tbx$ is called a
\emph{trend} of the time-series $\y$ and $\y-\tbx$ a \emph{cyclical
component} of $\y$.      

In Figure~\ref{fig:FF}, the cyclical components are on the lower
paletts, and the residues~(trends) are overlaid on the upper paletts.
\end{enumerate}
For the time-series of equicorrelation coefficient based on \cite{de},
\cite{wang20:_indus} applied a \emph{linear regression analysis} and
regarded the residue as the \emph{equi-correlation of industry return
(IEC)}. They claimed that ``We can see that the index displays prominent
countercyclical behavior in that it always rises during the recession
periods and declines during the expansion periods. This is expected
because economic recessions lead to greater comovement among stock
returns''. Fisher's z-transformation emphasizes the movement of the
time-series of equicorrelation coefficient~(ECC), and Hodrick-Prescott
filter improves the linear regression analysis of the ECC time-series.

As we see in Figure~\ref{fig:FF}, the trend
of z(ECC) time-series during period Dot-com bubble \& recession~(1998-01-02/2007-12-31), is more changing than that of
z(ECC) time-series after Lehman shock.
One may suppose that investors speculated
much to seemingly uncorrelated stocks during the Dot-com bubble.

\subsection{Binary multiple sequence alignment}\label{subsec:binary MSA}

We consider the binary \emph{multiple sequence alignment}~(MSA for short)
of an $\p$-residue (site) protein with $\n$ sequences where $\p=475$ and
$\n=2815$. The dataset is by the courtesy of Quadeer.  Quadeer et al.\@ did ``identify
groups of coevolving residues within HCV nonstructural protein 3 (NS3)
by analyzing diverse sequences of this protein using ideas from random
matrix theory and associated methods''~\cite[p.~7628]{querde}, and also found ``Sequence
analysis reveals three sectors of collectively evolving sites in
NS3. ..., there remained $\alpha=9$ eigenvalues greater than
$\lambda^{\mbox{rnd}}_{\mbox{max}}$, presumably representing intrinsic
correlations''\cite[p.~7631]{querde}. They detected
signals by a randomization from the data.

On the statistical model of \cite{querde}, 
Morales-Jimenez, one of the authors of \cite{quadeer18:_co_hiv_hcv},  commented ``the majority of variables
(protein positions in the genome) are essentially independent, and there
are just some small groups of variables which are correlated, giving
rise to the different spikes. These group of variables can be modeled
with equi-correlation, but the size of these groups is modeled as
fixed, i.e., not growing with the dimension of the protein. That leads
to a non-divergent spiked model, like the one considered in our Stat Sinica paper''
\cite{morales2021asymptotics}.
    
Nonetheless, for the dataset of the binary MSA~\citep{querde}, the broken-stick rule and the adjusted
correlation thresholding work well.
\begin{table}[ht]\centering
  \begin{tabular}{|c|c|c|c|c|c||c|c|c|}
   \hline
     & $\n$ & $\p/\n$    & $\lam{1}{\C}/\p$  &$\p$ & $\BS$ & $\ACT$ & Bai-Ng
   \\
\hline
  \small{ MSA }& 2815&0.1687 & 0.0216 & 475 & 3 & 10 & 4 \\
\hline
  \end{tabular}
 \caption{The binary MSA dataset.
  \label{tbl:MSA}}
\end{table}
\begin{itemize}
 \item  The number 3 of the sectors is detected by the broken-stick rule
	 $\BS(\C)=3$.

\item The number $\alpha=9$ of eigenvalues greater than $\lambda^{\mathrm{rnd}}_{\mathrm{max}}$ is close to $\ACT(\C)=10$.
      Here $\ACT$ is studied with the proportional limiting regime $\kr,\,\krr>0$.
\end{itemize}

To end this section about the broken-stick rule for the stock
return datasets and the binary MSA dataset, we observe:
Fama-French portfolio~(Dot-com bubble \& recession~(1998-01-02/2007-12-31)), and the binary MSA dataset fit a CLFM with $r=2$ but not an ACFM. The Fama-French portfolio~(after Lehman shock~(2010-01-04/2019-04-30)) fits both ACFM and a CLFM with $r=1$.

In conclusion, the pair of a CLFM and BS rule may be more useful than the pair of a model and ACT~\citep{FGZ22}, to estimate the number of factors of the binary MSA dataset; the number of factors estimated by our approach is the same as that estimated by \citep{querde}.
On the other hand, ACT may be useful for Fama-French portfolios, as they can predict the Fama-French factors.

In theory, to distinguish a model CLFM of $r=1$ from an ACFM, one may think of
 $\D=\diag(D_{11},\ldots,D_{\p\p})$. The diagonal entries $D_{ii}$ are asymptotically equal to the corresponding entries of $\bDelta$ (their ratios converge almost surely to 1 by Lemma~\ref{lem:D/Delta_identity}). And the diagonal entries of $\bDelta$, are just the Euclidean norm of rows of $\bLambda$.

Once we find a theoretically appropriate scaling and shifting of $D_{ii}$'s to uncover the concentration of $D_{ii}$'s, we could decide whether a given dataset fit a CLFM, an ACFM, or the generalization.

\section{Conclusion} \label{sec:conclusion}

We suppose the proportional limiting regime for factor models. We
established a general theorem for the asymptotic equispectrality of $\C$
and the naturally normalized sample covariance matrix for factor models.
Then, we derived the following assertions: for the introduced two models (CLFM and ACFM), the limiting largest bounded
eigenvalue and the limiting spectral distribution of the sample
correlation matrix $\C$, are scaling of those of $\C$ of the i.i.d.\@
case, with the scaling constant $1-\rho$. Here $\rho$ is ``a reduced
correlation coefficient''.  The largest eigenvalue of $\C$ divided by
the order $\p$ of $\C$ converges almost surely to $\rho$.

Notably, for an ACFM, the BS rule computes not the rank of the factor loading matrix $\bL$, but the essential rank
$(=1)$~(Theorem~\ref{thm:asympt_CLF}).  In other words, the
deterministic decreasing vector $\y=(\sum_{j=k}^K (j K)^{-1})_{k=1}^K$
of the BS rule examines an intrinsic structure of $\bL$, for an ACFM.
Moreover, $\y\in\Rset^K$ is the descending list of the limits of the
eigenvalues of $\bL\bL^\top\in\Rset^{\p\times\p}$ in $\p\to\infty$, for
a random factor loading matrix $\bL=\left[B_{ik}
\sqrt{y_{ik}}\right]_{\rangei, \rangek}$ where for each $i$ $(\rangei)$,
\begin{itemize}
 \item $\b_i=(B_{ik})_{k=1}^K$ is a random vector of
independent Rademacher
variables,  
\item $\y_i=(y_{ik})_{k=1}^K$ is such that $y_{ik}$ is the length of the
$k$th longest subintervals obtained by $K-1$ independent uniformly
distributed separators of $[0,\,1]$;
\end{itemize}
and $\b_1,\ldots,\b_\p,\ \y_1,\ldots,\y_\p$ are independent.  In this
case, $\bL$ does not satisfy the condition of an ACFM. A limit theory
for the order statistics of the eigenvalues of a sample
covariance/correlation matrix is awaited to analyze the BS
rule.

\bibliographystyle{agsm}\bibliography{Limits-of-BSrule}
\newpage\appendix
\setcounter{page}{1}
\pagenumbering{roman}

Supplementary material for the manuscript  \small ``Asymptotic locations of bounded and
unbounded eigenvalues of sample correlation
matrices of certain factor models - application to a components retention rule'' 

\bigskip
\bigskip
This supplementary article proves  Theorems~\ref{thm:perturbation_multiplicative},
\ref{thm:asymptotic equi-spectral},
\ref{thm:main_CLFM_largest}, 
\ref{thm:CLFM_BS}, and \ref{thm:asympt_CLF} of the main manuscript and collects some useful lemmas. The literature cited in this supplementary material are listed in References of the main text.

\section{Proofs of theorems} \label{sec:proof_theorems}

\subsection{Proof of Theorem~\ref{thm:perturbation_multiplicative}}
Proposition~\ref{prop:ineq_Fan_multiplication} clearly implies the latter inequality
    \begin{align}
       \sv{i}{\A\B}\le\sv{1}{\A}\sv{i}{\B}. \label{ineq:sAB_sAsB} 
    \end{align}
Now we prove the other inequality $\sv{m}{\A}\sv{i}{\B}\le \sv{i}{\A\B}$. If $\sv{m}{\A}=0$, then there is nothing to prove. We now assume that $\sv{m}{\A}>0$. Then it is necessary that $m\le n$ and the Hermitian matrix $\A\A^*$ is positive definite. Let $\A^+$ be the Moore-Penrose inverse of $\A$ defined as
$$\A^+:=\A^*(\A\A^*)^{-1}.$$
By \eqref{ineq:sAB_sAsB} and $\A\A^+=\I_m$,
\begin{align}
 \sv{i}{\B}=\sv{i}{\B\A\A^+}\le\sv{i}{\B\A}\sv{1}{\A^+}. \label{ineq:sA_sABsB+}
\end{align}
Noticing that 
$$\sv{1}{\A^+}=\sqrt{\lambda_1(\A^*(\A\A^*)^{-2}\A)}=\sqrt{\lambda_1((\A\A^*)^{-1})}=\frac{1}{\sv{m}{\A}},$$
we obtained the wanted inequality by multiplying $\sv{m}{\A}$ on both sides of \eqref{ineq:sA_sABsB+}.

\subsection{Proof of Theorem~\ref{thm:asymptotic equi-spectral}}
Note that $\C=(\D^{-1/2}\bDelta^{1/2})(\bDelta^{-1/2}\S\bDelta^{-1/2})(\bDelta^{1/2}\D^{-1/2})$ and the corresponding formula for $\tC$. By Theorem~\ref{thm:perturbation_multiplicative}, we have only  to prove that all the singular values of $\D^{-1/2}\bDelta^{1/2}$ and $\tD^{-1/2}\bDelta^{1/2}$ are concentrated near $1$, as follows:

\begin{lem}[Key]\label{lem:D/Delta_identity} Assume the same condition as Theorem~\ref{thm:asymptotic equi-spectral}. Namely, suppose that we are given a $K$-factor model with \protect{\ref{ass:Gamma_row_ne0}}-\protect{\ref{ass:karoui_cond_or_identity}}.
\begin{align*}
&\lim_\kr\snorm{\tD\bDelta^{-1}-\I_{\p} } =\lim_\kr\snorm{\D\bDelta^{-1} - \I_{\p} } =0&(a.s.).
\end{align*}
\end{lem}

We now give the proof of this lemma. Recall that $\C$ is invariant under scaling and shifting of data. In the proof, we assume that $\bmu=\o$ and the length of each row vector of $\bGamma$ is $1$: $\norm{\bgamma^i}=1$ for $i=1,\dots,\p$ without loss of generality. Then $\bDelta=\I_\p$. The almost sure convergence in the limiting regime $\kr,\,\krr$ is denoted by $\cas$. This proof is separated by two parts: the first assuming \ref{ass:karoui_cond_or_identity}(b), and the second assuming \ref{ass:karoui_cond_or_identity}(a). 

\begin{proof}[Proof under \protect{\ref{ass:karoui_cond_or_identity}}(b)]
Let $\bLambda=\diag(\sigma_1,\dots,\sigma_\p)$. We prove that $\snorm{\tD-\I}\cas 0$. It suffices to confirm
\begin{align}
\lim_\kr \max_{\rangei}\abs*{\frac{\norm{\x^i}^2}{\n}
 -1}&=0&(a.s.). \label{cnv:max_xi}
\end{align}
Let $\bL=[\ell_{ik}]_{\rangei,1\le k\le K}\in \Rset^{\p\times K}$. 
Because of $\norm{\bgamma^i}=1$, we have 
\begin{align}\sum_{k=1}^K \ell_{ik}^2 + \sigma_i^2=1.\label{bounded loading}
\end{align}
Thus
 \begin{align*}
 \frac{\norm{\x^i}^2}{\n}-1
 &=
\left( 
\begin{aligned}
 \summ\ell_{ik}^2
 \sumj\left(\frac{\fkt^2}{\n}-1\right)
 +\sum_{1\le k<k'\le K}2\ell_{ik}\ell_{ik'}
 \sumj\frac{\fkt{f}_{k'\j}}{\n}
\\
 + \summ{2\sigma_i\ell_{ik}\sumj}\frac{\fkt\eit}{\n}
+\sigma_i^2\left(\sumj\frac{\eit^2}{\n}-1\right)
\end{aligned}\right).
\end{align*}
Here $\max\left\{\ell_{ik}^2,\,2|\ell_{ik}\ell_{ik'}|,\,2|\sigma_i\ell_{ik}|,\,\sigma_i^2\right\}\le 1$ for $i=1,\dots,\p$, $1\le k\ne k'\le \n$. 
As a result, $\max_{{\rangei}}\abs*{\frac{\norm{\x^i}^2}{\n} - 1}$ is at most

\begin{dmath}
\summ\abs*{\sumj{\displaystyle\frac{\fkt^2}{\n}} - 1}
+ \displaystyle\sum_{1\le k<k'\le K}\abs*{
 \sumj{\frac{\fkt{f}_{k'\j}}{\n}}}
+\summ\max_\rangei\abs*{{\sumj}\frac{\fkt\eit}{\n}}
+\max_\rangei\abs*{\sumj\frac{ \eit^2}{\n} - 1}.
\label{eq:max_i|x^i|2/T-1}
\end{dmath} 

Then, for each possible $1\le k\le K$ and $1\le k<k'\le K$, we can prove
 that each of the four
 term of \eqref{eq:max_i|x^i|2/T-1} tends to 0 almost surely, as follows:
\begin{enumerate}[label=(\roman*)]
 \item $\set{\fkt^2\colon \j\ge1}$ is an array of i.\@i.\@d.\@ random variables with mean $1$. Use Lemma~\ref{lem:bai-yin2} with $\alpha=1, \beta=0$.
\item $\set{\fkt {f}_{k'\j} \colon\j\ge1}$ for $k\ne k'$ is an array of centered i.\@i.\@d.\@ random variables, since $\fkt$ and $f_{k',\j}$ are independent and
centered. Use Lemma~\ref{lem:bai-yin2} with $\alpha=1, \beta=0$.
\item $\set{\fkt\eit\colon\rangei,\,\rangej}$ is a double
array of centered i.\@i.\@d.\@ random variables, since both of $\fkt$ and $\eit$ are independent and centered. 
From $\var\fkt=\var\eit=1$ by Definition~\ref{def:K-FM}, we derive
$$\Exp|\fkt\eit|^2=\Exp|\fkt|^2\Exp|\eit|^2<\infty.$$
Use Lemma~\ref{lem:bai-yin2} with $\alpha=1, \beta=1$.
\item $\set{\eit^2\colon\rangei,\,\rangej}$ is a double
array of i.\@i.\@d.\@ random variables with mean $1$ and $\Exp|\eit|^4<\infty$ by \protect{\ref{ass:moment4}}. Use Lemma~\ref{lem:bai-yin2} with $\alpha=\beta=1$.
\end{enumerate}
Hence, \eqref{cnv:max_xi} is confirmed.

Now we prove that $\snorm{\D-\I}\cas 0$ as $\kr$. 
It suffices to guarantee
\begin{align}\lim_{\kr} \max_{\rangei}\abs*{\frac{\norm{{\x^i-{\bmx}^i}}^2}{\n} -1}&=0&(a.s.).
\label{cnv:max_xi_bmxi}
\end{align}
Let $\mx_i=\n^{-1}\sumj \xit$ be the sample average of the $i$th row of $\X$. Then by $\norm{\x^i -{\bmx}^i}^2 = \norm{\x^i}^2 - \n\abs*{\mx_i}^2$, 
\begin{align}
\max_{{\rangei}}\abs*{\frac{\norm{\x^i-\bmx^i}^2}{\n}-1}\le \max_{{\rangei}}\abs*{\frac{\norm{\x^i}^2}{\n}-1}+\max_{{\rangei}}\abs*{\mx_i}^2. \label{eq:max_xi_bmxi}
\end{align}
The first term of the right side converges almost surely to 0, by
 \eqref{cnv:max_xi}.

As for the second term $\max_{{\rangei}}\abs*{\mx_i}^2$, from the definition of $\mx_i$,  $|\ell_{ik}|\le 1$,  and $|\sigma_i|\le 1$, we get
\begin{align*} 
\max_{\rangei}\abs*{\mx_i} \le \summ \abs*{\sumj\frac{\fkt }{\n}} +
 \max_{\rangei} \abs*{\sumj\frac{\eit}{\n}}.
\end{align*}
The first term of the right side converges almost surely to 0, as for each $1\le k\le K$, $\set{ \fkt\colon \j\ge1}$ is an array of centered i.\@i.\@d.\@ random variables. The second term of the right side converges almost surely to 0 by Lemma~\ref{lem:bai-yin2}, since $\set{\eit\colon\rangei,\,\rangej}$ is a double array of i.\@i.\@d.\@ random variables. Therefore the right side of \eqref{eq:max_xi_bmxi} converges to $0$ almost surely. Consequently, \eqref{cnv:max_xi_bmxi} is guaranteed. 
\end{proof}

\begin{proof}[Proof under \ref{ass:karoui_cond_or_identity}(a)] 
This part is inspired by the proof of Lemma~4 in \cite{karoui09concentration}. 
First, we proceed the truncation on the random variables $\eit$. 
Let 
$$\mathbf{T}_\p=\left[\eit\Indctr{|\eit|<\sqrt{\frac{\n}{(\log\n)^{1+\varepsilon}}}}\right]_{i,\j}$$
be the $\p\times \n$ matrix with truncated entries, where `$\mathrm{I}$' is the indicator function and $\varepsilon>0$. Then we re-center the entries of $\mathbf{T}_\p$ by defining
$$\hbPsi=\mathbf{T}_\p-e_\n \ones{\p\times \n},$$
where $$e_\n=\Exp\left(\eit\Indctr{|\eit|<\sqrt{\n/(\log\n)^{1+\varepsilon}}}\right)$$ and $\ones{\p\times\n}$ is the $\p\times\n$ matrix of unity. Let $\heit$ be the entries of $\hbPsi$.

We will prove that after replacing $\bPsi$ with $\hbPsi$, the diagonal entries of $\tD$ and $\D$ do not change a lot. 
By \ref{ass:karoui_cond_or_identity} and \cite[Lemma~2]{karoui09concentration}, almost surely for large enough $\p$, we have $\bPsi=\mathbf{T}_\p$, so almost surely, for large enough $\p$, the truncation does not modify $\D$ and $\tD$. Also $\D$ is invariant under the translation, so the centering does not impact $\D$. 
For $\tD$, because $\eit$ is centered and has finite forth moment by \ref{ass:moment4},  we get
$$\left(\frac{\n}{(\log\n)^{1+\varepsilon}}\right)^{3/2} |e_\n| \ \le\ \Exp \left(|\eit|^4\Indctr{|\eit|\ge\sqrt{\frac{\n}{(\log\n)^{1+\varepsilon}}}}\right).$$
Here the right side is $O(1)$.
Thus $e_\n=o(\n^{-1})$. In the limiting regime $\kr,\krr$,
$$\snorm{\bPsi-\hbPsi}= e_\n \snorm{\ones{\p\times \n}} =  e_\n \sqrt{\p \n} \,\cas0.$$
Let $\hbx^i$ be the $i$th row of data after truncation and re-centering. 
As we have assumed $\bmu=\o$ without loss of generality, we can write
$$\begin{bmatrix}\hbx^1\\ \vdots \\ \hbx^\p       
      \end{bmatrix}=
       \bL\bF + \bLambda\hbPsi.$$
 Let $\bg^i$ be the $i$th row vector of $\bLambda$. 
Since we assumed the length of each row vector of $\bGamma$ is 1 at the beginning of this proof, we get $\norm{\bell^i}\le1$ and $\norm{\bg^i}\le 1$. 
Hence
\begin{align*}
   \max_\rangei\left|\frac{\norm{\hbx^i}^2-\norm{\x^i}^2}{\n}\right| &\le \frac{1}{\n}\max_\rangei\left|2\bell^i\bF(\bPsi-\hbPsi)^\top{(\bg^i)}^\top\right|+\frac{1}{\n}\max_\rangei\left|\bg^i(\bPsi\bPsi^\top-\hbPsi\hbPsi^\top){(\bg^i)}^\top\right| \\
   & \le \frac{2}{\n}\snorm{\bF}\snorm{\bPsi-\hbPsi}+\frac{1}{\n}\snorm{\bPsi-\hbPsi}\snorm{\bPsi^\top}+\frac{1}{\n}\snorm{\hbPsi}\snorm{\bPsi^\top-\hbPsi^\top} \\
    & \cas 0.
\end{align*}

Now define the function $d_i:\Rset^{K\n}\times \Rset^{\p\n}\xrightarrow{}[0,\,\infty)$ by 
$$d_i(\bF,\bPsi):=\norm{\bell^i\bF+\bg^i\bPsi}.$$
From the proof of \cite[Lemma~4]{karoui09concentration}, $d_i$ is convex
 and $1$-Lipschitz (as $\norm{\bgamma^i}=1$).
 Let $\mu_f$ and
 $\mu_{\hpsi}$ be the distribution of $f_{k\j}$ and $\heit$,
 respectively. Conditioning on $\bF$ (i.e., for every fixed $\bF$), we
 apply \cite[Corollary~4.10, pp.77--78,
 Eq.(4.10)]{ledoux2001concentration} on the variables $\heit$. By
 $|\heit|\le \sqrt{\n/(\log \n)^{1+\varepsilon}}$, it holds that for all $\delta>0$,
\begin{align}
\int_{\Rset^{\p\n}}\Indctr{\,\abs*{\frac{d_i(\bF,\hbPsi)}{\sqrt{\n}}-\MedianF\left(\frac{d_i(\bF,\hbPsi)}{\sqrt{\n}}\right)}>\delta}\dd\mu_{\hpsi}^{\otimes\p\n}(\hbPsi)&\le 4\exp\left(-\frac{\delta^2(\log\n)^{1+\varepsilon}}{16}\right) . \label{eq:cond_talagrand}
\end{align}
Here $\MedianF$~($\ExpF$, $\varF$, resp.) represents \emph{a} conditional median~(the conditional expectation, the conditional variance, resp.) knowing $\bF$. 
By \cite{10.1214/aop/1176996411}, $\MedianF(\fra{d_i(\bF,\hbPsi)}{\sqrt{\n}})$ is measurable, so we can integrate \eqref{eq:cond_talagrand}  with  $\mu_f^{\otimes K\n}$ by $\bF$.
Then, we get, for all $\delta>0$,
\begin{align*}
\P\left(\abs*{\frac{d_i(\bF,\hbPsi)}{\sqrt{\n}}-\MedianF\left(\frac{d_i(\bF,\hbPsi)}{\sqrt{\n}}\right)}>\delta\right)&\le 4\exp\left(-\frac{\delta^2(\log\n)^{1+\varepsilon}}{16}\right).
\end{align*}
Readily, 
\begin{align*}\P\left(\max_\rangei\abs*{\frac{d_i(\bF,\hbPsi)}{\sqrt{\n}}-\MedianF\left(\frac{d_i(\bF,\hbPsi)}{\sqrt{\n}}\right)}>\delta\right)&\le 4\p\exp\left(-\frac{\delta^2(\log\n)^{1+\varepsilon}}{16}\right).
\end{align*}
Borel-Cantelli's lemma yields
\begin{align}
    \max_\rangei\left|\frac{d_i(\bF,\hbPsi)}{\sqrt{\n}}-\MedianF\left(\frac{d_i(\bF,\hbPsi)}{\sqrt{\n}}\right)\right|\cas 0. \label{eq:d_i_to_median}
\end{align}

On the other hand, by \cite{mallows91:_anoth_ocinn}, and then by applying \cite[Proposition~1.9]{ledoux2001concentration} to $\fra{d_i(\bF,\hbPsi)}{\sqrt{\n}}$ with inequality \eqref{eq:cond_talagrand}, we have
\begin{align}
    & \left|\MedianF\left(\frac{d_i(\bF,\hbPsi)}{\sqrt{\n}}\right)-\ExpF\left(\frac{d_i(\bF,\hbPsi)}{\sqrt{\n}}\right)\right|\le \sqrt{\varF\left(\frac{d_i(\bF,\hbPsi)}{\sqrt{\n}}\right)}
   \le C(\log\n)^{-(1+\varepsilon)/2}\label{ineq:median mean}
\end{align}
where $C$ is an absolute constant independent of $i$.
By \eqref{bounded loading} and the first two arguments (i, ii) just below \eqref{eq:max_i|x^i|2/T-1}, we can check 
$$\max_\rangei\left|\frac{1}{\n}\norm{\bell^i\bF}^2-\norm{\bell^i}^2\right|\cas 0.$$
Moreover, by Lebesgue's dominated convergence theorem, the premise $\var\psi_{11}=1$ deduces $\var{\hpsi_{11}}\to1$.
Hence 
\begin{align*}
&\max_\rangei\left|\ 
\abs*{
\left(\ExpF\frac{d_i(\bF,\hbPsi)}{\sqrt{\n}}\right)^2
-1}
-
\varF\frac{d_i(\bF,\hbPsi)}{\sqrt{\n}}
\ \right|
\\
&\le\max_\rangei\left|
\varF\frac{d_i(\bF,\hbPsi)}{\sqrt{\n}}
+
\left(\ExpF\frac{d_i(\bF,\hbPsi)}{\sqrt{\n}}\right)^2
-1\right|
=\max_\rangei\left|\ExpF\frac{d_i^2(\bF,\hbPsi)}{\n}-1\right|
\\
&\le \max_\rangei\left\{\left|\frac{1}{\n}\norm{\bell^i\bF}^2-\norm{\bell^i}^2\right|+\abs*{\var(\hpsi_{11})-1}\,\norm{\bg^i}^2\right\} 
\\
&\cas 0.\end{align*}
As a result, by the second inequality in \eqref{ineq:median mean}, 
\begin{align*}
\max_\rangei
\abs*{
\ExpF\frac{d_i(\bF,\hbPsi)}{\sqrt{\n}}
-1}\le \max_\rangei
\abs*{
\left(\ExpF\frac{d_i(\bF,\hbPsi)}{\sqrt{\n}}\right)^2
-1}\cas 0.
\end{align*}
Again by\eqref{ineq:median mean}, 
\begin{align*}
 \max_\rangei\left|\MedianF\left(\frac{d_i(\bF,\hbPsi)}{\sqrt{\n}}\right)-1\right|\cas 0,
\end{align*}
from which \eqref{eq:d_i_to_median} concludes
$$\max_\rangei\left|\frac{d_i(\bF,\hbPsi)}{\sqrt{\n}}-1\right|\cas 0.$$

For $\D$, we use the similar arguments as \cite{karoui09concentration} with the above adaptions. We omit the details. The proof is complete. 
\end{proof}

\subsection{Proof of \protect{Theorems~\ref{thm:main_CLFM_largest}}, \protect{\ref{thm:CLFM_BS}} and \protect{\ref{thm:asympt_CLF}}}
\label{prf:prop_main}

To establish Theorem~\ref{thm:main_CLFM_largest}, by Theorem~\ref{thm:perturbation_multiplicative} and  $\bDelta=(L^2+\sigma^2)\I$, it is enough to determine the asymptotic locations of largest eigenvalues of sample covariance matrix $\S$ of a CLFM.
\begin{thm}\label{thm:paraphrased main theorem} If a CLFM satisfies \ref{ass:L_rank_r_sr_infty}, then for $k=1,\dots,r$,
  \begin{align}
   & \lim_\kr \frac{\lambda_{k}(\tS)}{\lambda_k(\bL\bL^\top)} = \lim_\kr \frac{\lam{k}{\S}}{\lambda_k(\bL\bL^\top)} = 1 & (a.s.), \label{eq:para_main_CLFM_unbounded}\\
   & \lim_\kr \lambda_{r+1}(\tS)=\lim_\kr \lambda_{r+1}(\S) = \sigma^2(1+\sqrt{c})^2 & (a.s.). \label{eq:para_main_CLFM_bounded}
\end{align}
\end{thm}
\begin{proof}[Proof of \protect{\eqref{eq:para_main_CLFM_unbounded}}]
Let $1\le k\le r$. For $\tS$, by the definition \eqref{eq:def_tS} and the model setting, we observe
$$\sqrt{\lambda_k(\tS)}=\sv{k}{\frac{1}{\sqrt{\n}}\X}=\frac{1}{\sqrt{\n}}\sv{k}{\bL\bF+\sigma\bPsi}.$$
By Corollary~\ref{corol:small_size_pert_singular_v},
$$\abs*{\sqrt{\lambda_k(\tS)}-\frac{1}{\sqrt{\n}}\sv{k}{\bL\bF}}\le \frac{\sigma}{\sqrt{\n}}\sv{1}{\bPsi}.$$
By dividing the above inequality with $\sv{k}{\bL}$, 
\begin{align}
&\left|\sqrt{\frac{\lambda_k(\tS)}{\lambda_k(\bL\bL^\top)}}-\frac{1}{\sqrt{\n }}\frac{\sv{k}{\bL\bF}}{\sv{k}{\bL}}\right|
\le \frac{\sigma}{\sv{k}{\bL}}\sv{1}{\frac{\bPsi}{\sqrt{\n}}}
.\label{ineq:diff_sv_ratios}
\end{align}
Here $\lim_{\kr}1/\sv{k}{\bL}=0$ by \ref{ass:L_rank_r_sr_infty}.
For $\bPsi\in\Rset^{\p\times\n}$, $\lim_\kr\sv{1}{\n^{-1/2}\bPsi}=(1+\sqrt{\c})^2$ (a.s.) by
Proposition~\ref{prop:lambda1C}~\eqref{noncentered}. Thus $\lim_{\kr}\eqref{ineq:diff_sv_ratios}=0$ (a.s.).

Now we only need to verify that
\begin{align}
    \frac{1}{\sqrt{\n}}\frac{\sv{k}{\bL\bF}}{\sv{k}{\bL}}&\cas 1&(\kr).\label{eq:lim_LF_over_F_to_1}
\end{align}
By applying the strong law of large numbers to each element of $\n^{-1}\bF\bF^\top\in\Rset^{K\times K}$, we deduce
\begin{align*}
    \frac{1}{\n}\bF\bF^\top\cas \I_K, 
\end{align*}
which means that $\sv{1}{\bF/\sqrt{\n}}\cas 1$, and $\sv{K}{\bF/\sqrt{\n}}\cas
 1$. 
Then, Theorem~\ref{thm:perturbation_multiplicative} verifies 
 \eqref{eq:lim_LF_over_F_to_1}. Therefore, we established $\lim_\kr\fra{\lambda_{k}(\tS)}{\sv{k}{\bL}} = 1 $ (a.s.) for $\rangekr$ and $K\ne 0$.

 \medskip
From the above proof, we obtain a proof of the latter identity of
 \eqref{eq:para_main_CLFM_unbounded} for $\rangekr$ and $K\ne 0$, through the following modification. Firstly
 note that $\mX=(L^2+\sigma^2)^{-1/2}(\bL\mbF +\sigma\mbPsi)$ where $\mbF=\left[\n^{-1}
 \sumj \fkt \right]_{\k,\j}$ and $\mbPsi=\left[\n^{-1} \sumj \eit
 \right]_{i,\j}$. Then, in the proof of $\lim_\kr\fra{\lambda_k(\tS)}{\sv{k}{\bL}} =1$ (a.s.), we use  $\bF - \mbF$, $\bPsi - \mbPsi$, and Proposition~\ref{prop:lambda1C}~\eqref{centered},  
instead of $\bF$, $\bPsi$, and Proposition~\ref{prop:lambda1C}~\eqref{noncentered}.
\end{proof}

\begin{proof}[Proof of \eqref{eq:para_main_CLFM_bounded}] 
By Lemma~\ref{lem:LSD}, we obviously have
 \begin{align*}
   \liminf_\kr\lambda_{r+1}(\tS)&\ge \sigma^2 (1+\sqrt{\c})^2&(a.s.).
\end{align*} 
Thus we only need to assure that
\begin{align}\label{ub_of_limsup_of_r+1th_eig_of_tS}
\limsup_\kr\lambda_{r+1}(\tS)&\le \sigma^2(1+\sqrt{\c})^2 &(a.s.).
\end{align}
When $r=0$, \eqref{ub_of_limsup_of_r+1th_eig_of_tS} follows from Proposition~\ref{prop:lambda1C}~\eqref{noncentered}. So we assume now $r>0$. The matrices $\tS$ can be written as
$$\tS=\frac{1}{\n}(\bL\bF + \sigma\bPsi)(\bL\bF + \sigma\bPsi)^\top.$$
Thus using Proposition~\ref{prop:inequality_Fan_Ky} with $i=0,j=r$, and noting that $s_{r+1}(\bL\bF)=0$, we have
$$\lambda_{r+1}(\tS)=\frac{1}{\n} s^2_{r+1}(\bL\bF + \sigma^2\bPsi)\le \frac{\sigma^2}{\n}\lambda_1(\bPsi\bPsi^\top).$$
As $\kr$, taking the limsup, by Proposition~\ref{prop:lambda1C}, \eqref{noncentered}, we conclude \eqref{ub_of_limsup_of_r+1th_eig_of_tS}.

Finally, the proof of $\lim_\kr \lamLC{\S}=\sigma^2(1+\sqrt{c})^2$ (a.s.) is that of
$\lim_\kr
 \lambda_{r+1}(\tS)=\sigma^2(1+\sqrt{c})^2$ (a.s.), except that  $\S$,
 $\bF$, $\bPsi$, and  Proposition~\ref{prop:lambda1C}~\eqref{noncentered} are superseded by $\tS$,
$\bF-\left[\n^{-1}\sumj\fkt\right]_{k,\j}$, $\bPsi-\left[\n^{-1}\sumj \eit\right]_{i,\j}$, and  Proposition~\ref{prop:lambda1C}~\eqref{centered},
 respectively.
\end{proof}

\begin{proof}[Proof of Theorem~\ref{thm:CLFM_BS}]
It is well-known that $\harmonic{r} =\log\p + O(1)$ for any fixed $r\ge
 1$. By Theorem~\ref{thm:main_CLFM_largest} along with condition
 \eqref{eq:lambda_r_underbound}, it holds almost surely that
$\lamLC{\C}=(1-\rho)(1+\sqrt{\c})^2+o(1)$ and
 $\displaystyle\varliminf_{\p\to\infty}(\lambda_r(\C)/\log(\p))>1$. As a result,
 for sufficiently
large $\p$, it follows almost surely that $\lam{k}{\C}>\harmonic{r}$ $(\rangekr)$,  and
$\lamLC{\C}\le \harmonic{r+1}$.  Thus,
$\lim_\kr\BS(\C)=r$ (a.s.). 

We can demonstrate $\lim_\kr \BS(\tC)=r$ (a.s.) mutatis mutandis. \end{proof}

Now we prove Theorem~\ref{thm:asympt_CLF}. For an asymptotic CLF, we establish the following result for the largest and the second largest eigenvalues of sample correlation matrices $\tC$ and $\C$, and then we proceed in the same manner as proving Theorem~\ref{thm:CLFM_BS}.
\begin{thm} Under the same conditions as Theorem~\ref{thm:asympt_CLF}, for $\M=\C,\tC$, the following hold almost surely:
\begin{align*}
\lim_{\kr}\frac{\lambda_1(\M)}{\p}&=\frac{\norm{\bell}^2}{\norm{\bell}^2+\sigma^2},\quad
&\lambda_2(\M)&=(1-\rho)(1+\sqrt{c})^2+o(\log(\p))&(\kr).
\end{align*}
\end{thm}
\begin{proof} After normalizing the $k$th row of $\X$ with $\sqrt{\norm{\bell^k}^2+\sigma_k^2}$, we may assume without loss of generality that $\norm{\bell^k}^2+\sigma_k^2=1$ for each $k\in [1,\p]$. Then $\norm{\bell}^2+\sigma^2=1$. 
For
$$\bL_0=\begin{bmatrix}
    \bell^\top & \dots & \bell^\top 
\end{bmatrix}^\top,$$
Corollary~\ref{corol:small_size_pert_singular_v} implies
\begin{align}
    \abs*{s_i(\bL)-s_i(\bL_0)}\le \snorm{\bL-\bL_0} \le \sqrt{\sum_{k=1}^N \norm{\bell^k-\bell}^2}=o(\log^{1/2}(N)). \label{eq:si_bL_bL0}
\end{align}

Using Theorem~\ref{thm:asymptotic equi-spectral}, Corollary~\ref{corol:small_size_pert_singular_v} and Theorem~\ref{thm:perturbation_multiplicative}, we have, as $\kr,\krr>0$, almost surely,
\begin{align*}\sqrt{\frac{\lambda_1(\tC)}{\p}}
 & \sim \frac{1}{\sqrt{TN}}s_1(\bL\bF+\bLambda \bPsi)  = \frac{1}{\sqrt{TN}}s_1(\bL\bF)+ O(\frac{1}{\sqrt{TN}}s_1(\bLambda \bPsi)) \\
 & \sim \frac{1}{\sqrt{N}}s_1(\bL) + O(\frac{1}{\sqrt{N}}) 
 \sim  \frac{1}{\sqrt{N}}s_1(\bL_0) + O(\frac{1}{\sqrt{N}}) \xrightarrow[]{} \norm{\bell}.
\end{align*}

As a result of $s_2(\bL_0)=0$ and \eqref{eq:si_bL_bL0}, we get $s_2(\bL)=o(\log^{1/2} N)$. The same sequence of inequalities applying to the second largest eigenvalue of $\tC$, we have
\begin{align*}\sqrt{\lambda_2(\tC)}
 & \sim \frac{1}{\sqrt{T}}s_2(\bL\bF+\bLambda \bPsi)  = \frac{1}{\sqrt{T}}s_2(\bL\bF)+ O(\frac{1}{\sqrt{T}}s_2(\bLambda \bPsi)) \\
 & \sim s_2(\bL) + O(1)  = o(\log^{1/2} N).
\end{align*}
    For $\C$, the proof is similar.
\end{proof}
\section{Some useful lemmas}\label{sec:useful lemma}

The following Weyl's inequality is well-known~\cite[Theorem~4.3.1]{horn2013}.
\begin{prop}[Weyl's inequality]\label{lem:Weyl's ineq} Let 
 $\bN$ and $\H$ be $\p\times\p$ Hermitian matrices, and let $\M=\bN+\H$. Then
\begin{align*}&\lambda_j(\bN)+\lambda_k(\bH) \le \lambda_i(\bM) \le \lambda_r(\bN)+\lambda_s(\bH) & (r+s-1\le i\le j + k -\p).
\end{align*}
\end{prop}

An analogous theorem on the singular values of matrices can be established.
\begin{prop}[\protect{\cite[Theorem~A.8]{bai2010spectral}}]\label{prop:inequality_Fan_Ky}
Let $\A$ and $\B$ be two $p\times n$ complex matrices. Then for any nonnegative integers $i,j$,
\begin{align*} \sv{i+j+1}{\A+\B}\le\sv{i+1}{\A}+\sv{j+1}{\B}. 
\end{align*}
\end{prop}

The conjugated transpose of a complex matrix $\A$ is denoted by $\A^*$.
\begin{proof}
 Apply the second inequality $\lambda_i(\bM) \le \lambda_r(\bN)+\lambda_s(\bH)$ of Proposition~\ref{lem:Weyl's ineq} to Hermitian matrices $\bN=\begin{bmatrix} & \A \\ \A^* &	\end{bmatrix}$ and  $\H=\begin{bmatrix} & \B\\ \B^* & \end{bmatrix}$. 
\end{proof}
Using this proposition, we establish the small-size perturbation
inequality for singular values.
\begin{corol}%[\protect{\cite[Corollary~7.3.5~(a)]{horn2013}}]
\label{corol:small_size_pert_singular_v}
Let $\A$ and $\B$ be the same as in Proposition~\ref{prop:inequality_Fan_Ky}. Then for $i\ge 1$,
    \begin{align*} |\sv{i}{\A+\B}-\sv{i}{\A}|\le \sv{1}{\B}.
\end{align*}
\end{corol}
\begin{proof}
In Proposition~\ref{prop:inequality_Fan_Ky} taking $j=0$ and replacing $i+1$ by $i$, we get 
$\sv{i}{\A+\B}-\sv{i}{\A}\le\sv{1}{\B}.$
The other part is checked from the above inequality by
$\sv{i}{\A+\B-\B}-\sv{i}{\A+\B}\le\sv{1}{-\B}$ and note that $\sv{1}{-\B}=\sv{1}{\B}$.
\end{proof}

\begin{prop}[\protect{\cite[Theorem~A.10]{bai2010spectral}}]\label{prop:ineq_Fan_multiplication}Let $\A$ and $\B$ be complex matrices of order $m\times n$ and $n\times p$. For any $i, j \ge 0$, we have
$$\sv{i+j+1}{\A\B}\le \sv{i+1}{\A}\sv{j+1}{\B}$$
\end{prop}

To argue forms of the law of large numbers in the proportional limiting
regime, we make use of the following:
\begin{lem}[\protect{\cite[Lemma~2]{bai-yin2}}]\label{lem:bai-yin2}
Suppose that $\set{y_{ij} \colon\, i,j\ge1}$
is a double array of \emph{i.\@i.\@d.\@} random variables. Let $\alpha>{1}/{2}$, $\beta\ge 0$ and $Q>0$ be constants. Then,
\begin{align*}
& \lim_{n\to\infty}\max_{i\le Q n^\beta}\abs*{\sum_{j=1}^n \frac{y_{ij}-m}{n^\alpha}}= 0\quad (a.s.)\\
\iff&
\Exp\abs{y_{11}}^{\frac{1+\beta}{\alpha}}< \infty\; \&
\;  m=\begin{cases}\displaystyle
\Exp y_{11},&(\alpha\le 1),\\
\displaystyle\text{\rm any},&(\alpha>1).
\end{cases}
\end{align*}
\end{lem}

For all square matrices $\A$ and $\B$ of order $n$, the characteristic
polynomial of $\A\B$ is that of $\B\A$.  For two matrices $\A$ and $\B$
of size $n \times m$ and $m \times n$, respectively, the matrix $\A\B$
has the same nonzero eigenvalues as $\B\A$.

In the following, the first assertion is due to
\cite[Theorem~3.1]{YBK88} and the second to
\cite[(2.7)]{10.2307/25053330}.
\begin{prop} \label{prop:lambda1C} Let $\Z=\left[ z_{i\j}\right]_{i,\j}\in\Rset^{(K+\p)\times\n}$ satisfy the following subarray condition:
There is an infinite  matrix $\mathcal{Z}=[\zit]_{i\ge 1,\j\ge 1}$ such that
	\begin{itemize}
	 \item 
all the entries of $\mathcal{Z}$ are centered \emph{i.\@i.\@d.\@} real random variables having unit variances and the finite fourth moments; and
	 \item for each $\p\ge1$, the matrix $\Z=\left[\zit\right]_{1\le i\le
	K+\p,\,\rangej}$ is the
	top-left submatrix of $\mathcal{Z}$ where $\n=\n(\p)$.
	\end{itemize}
Moreover,
let $\mZ=\left[\n^{-1}\sumj z_{i\j}\right]_{i,\j}\in\Rset^{(K+\p)\times\n}$.
Then
\begin{align}
\lim_\kr \lam{1}{\frac{1}{\n}\Z\Z^\top}&= (1+\sqrt{\c})^2&
 (a.s.). \label{noncentered}\tag{1}
 \\
 \lim_{\n\to\infty}\sv{1}{\frac{1}{\sqrt\n}(\Z-\mZ)}&=1+\sqrt{\c}&(a.s.).\label{centered}\tag{2}
\end{align}
\end{prop}

By the \emph{empirical spectral measure} of an
$\p\times\p$ semi-definite matrix $\H$, we mean a probability measure $\theta$ such that
\begin{align*}
\theta(A)&=\frac{\#\set{i\in[1,\,\p] \colon\, \lam{i}{\H}\in A}}{\p}&(A\subseteq\Rset). 
\end{align*}
If the \emph{empirical spectral measure} of $\H$ converges weakly to a probability measure $\vartheta$ in a
given limiting regime, we call $\vartheta$ the \emph{limiting spectral measure} of $\H$.

\emph{Mar\v{c}enko-Pastur probability measure} of index $c>0$ and scale parameter $s>0$ has the
probability density function
\begin{align*}
p(x)= \begin{cases}
  \frac{1}{2\pi x c s}\sqrt{(a_+-x)(x-a_-)} &(a_-\le x\le a_+),\\
  0&(\mbox{otherwise})
 \end{cases}
\end{align*}
with an additional point mass of value $\max\{0,1-1/c\}$ at the origin $x=0$,
where $a_\pm=s(1\pm\sqrt{c})^2$. $a_+$ is called the \emph{right-edge}.

\begin{lem}\label{lem:LSD}
In every CLFM,  if ${{\p,\n\to\infty,\ \krr}}\in(0,\,\infty)$,
then all the limiting spectral measures of $\S$ and $\tS$ are 
 Mar\v{c}enko-Pastur probability measure of index $c$
 and scale parameter $\sigma^2$.
\end{lem}
\begin{proof}
Proved likewise as in~\cite[Section~7]{Akama23:_correl}.
\end{proof}
\end{document}